\documentclass[12pt,oneside,reqno]{amsart}

\usepackage{amsmath,amssymb,amsthm,textcomp}
\usepackage{amsfonts,graphicx}
\usepackage[mathscr]{eucal}
\usepackage{float}
\usepackage{color}
\usepackage{csquotes}
\usepackage[backend=bibtex,%
firstinits=true,%
doi=false,%
isbn=true,%
url=false,%
maxnames=99]{biblatex}%

\AtEveryBibitem{\clearfield{issn}}
\AtEveryCitekey{\clearfield{issn}}
\numberwithin{equation}{section}
\DeclareNameAlias{sortname}{last-first}
\theoremstyle{definition}
\usepackage{mathtools}
\numberwithin{equation}{section}

\newcommand{\ncom}{\newcommand}

\ncom{\beq}{\begin{equation}}
	\ncom{\eeq}{\end{equation}}
\ncom{\bea}{\begin{eqnarray*}}
	\ncom{\eea}{\end{eqnarray*}}
\ncom{\beqa}{\begin{eqnarray}}
	\ncom{\eeqa}{\end{eqnarray}}
\ncom{\nno}{\nonumber}
\ncom{\non}{\nonumber}
\ncom{\ds}{\displaystyle}
\ncom{\half}{\frac{1}{2}}
\ncom{\mbx}{\makebox{.25cm}}
\ncom{\hs}{\mbox{\hspace{.25cm}}}
\ncom{\rar}{\rightarrow}
\ncom{\Rar}{\Rightarrow}
\ncom{\noin}{\noindent}
\ncom{\bc}{\begin{center}}
	\ncom{\ec}{\end{center}}
\ncom{\sz}{\scriptsize}
\ncom{\rf}{\ref}
\ncom{\s}{\sqrt{2}}
\ncom{\sgm}{\sigma}
\ncom{\Sgm}{\Sigma}
\ncom{\psgm}{\sigma^{\prime}}
\ncom{\dt}{\delta}
\ncom{\Dt}{\Delta}
\ncom{\lmd}{\lambda}
\ncom{\Lmd}{\Lambda}
\ncom{\Th}{\Theta}
\ncom{\e}{\eta}
\ncom{\eps}{\epsilon}
\ncom{\pcc}{\stackrel{P}{>}}
\ncom{\lp}{\stackrel{L_{p}}{>}}
\ncom{\dist}{{\rm\,dist}}
\ncom{\sspan}{{\rm\,span}}
\ncom{\re}{{\rm Re\,}}
\ncom{\im}{{\rm Im\,}}
\ncom{\sgn}{{\rm sgn\,}}
\ncom{\ba}{\begin{array}}
	\ncom{\ea}{\end{array}}
\ncom{\hone}{\mbox{\hspace{1em}}}
\ncom{\htwo}{\mbox{\hspace{2em}}}
\ncom{\hthree}{\mbox{\hspace{3em}}}
\ncom{\hfour}{\mbox{\hspace{4em}}}
\ncom{\vone}{\vskip 2ex}
\ncom{\vtwo}{\vskip 4ex}
\ncom{\vonee}{\vskip 1.5ex}
\ncom{\vthree}{\vskip 6ex}
\ncom{\vfour}{\vspace*{8ex}}
\ncom{\norm}{\|\;\;\|}
\ncom{\integ}[4]{\int_{#1}^{#2}\,{#3}\,d{#4}}
\ncom{\vspan}[1]{{{\rm\,span}\{ #1 \}}}
\ncom{\dm}[1]{ {\displaystyle{#1} } }
\ncom{\ri}[1]{{#1} \index{#1}}

\newtheorem{theorem}{\bf Theorem}[section]
\newtheorem{remark}{\bf Remark}[section]
\newtheorem{proposition}{Proposition}[section]
\newtheorem{lemma}{Lemma}[section]
\newtheorem{corollary}{Corollary}[section]

\newtheoremstyle
{remarkstyle}
{}
{11pt}
{}
{}
{\bfseries}
{:}
{     }
{\thmname{#1} \thmnumber{#2} }

\theoremstyle{remarkstyle}

\def\eps{\varepsilon}

\begin{document}
\title{On multidimensional elephant random walk with stops and random step sizes}
	
\author[Shyan Ghosh]{Shyan Ghosh}
\address{Shyan Ghosh, Department of Mathematics, Indian Institute of Technology Bhilai, Durg, 491002, India.}
\email{shyanghosh@iitbhilai.ac.in}
\author[Manisha Dhillon]{Manisha Dhillon}
\address{Manisha Dhillon, Department of Mathematics, Indian Institute of Technology Bhilai, Durg, 491002, India.}
\email{manishadh@iitbhilai.ac.in}
\author[Kuldeep Kumar Kataria]{Kuldeep Kumar Kataria}
\address{Kuldeep Kumar Kataria, Department of Mathematics, Indian Institute of Technology Bhilai, Durg, 491002, India.}
\email{kuldeepk@iitbhilai.ac.in}
\date{\today}
	
\begin{abstract}
In this paper, we study the number of moves in a multidimensional elephant random walk with stops. We establish several convergence results for the number of moves, including the law of large numbers and the law of iterated logarithm. Using a martingale approach, we study the multidimensional elephant random walk with random step sizes. For this model, we obtain several almost sure convergence results for the number of moves, including the law of large numbers, the quadratic strong law, the law of iterated logarithm and the central limit theorem. Similar convergence results are derived for the multidimensional elephant random walk with random step sizes.
\end{abstract}

\maketitle

\section{Introduction}
Random walks are extensively studied models that have various applications in physics, biology, finance, \textit{etc.} Many physical and biological systems are influenced by long-term dependence, and require models that incorporate memory effects. However, such memory effects can not be captured by standard random walks as they are Markovian, and have memoryless property. To overcome these limitations, Sch\"utz and Trimper (2004) introduced and studied a discrete-time random walk, namely, the elephant random walk (ERW) which has complete memory of its past. At each step, it is generated by recalling a randomly selected step from the past and modifying it according to a fixed parameter leading to its non-Markovian behaviour. 

Let us briefly describe the standard ERW. It is a one-dimensional discrete-time random walk on $\mathbb{Z}$ in which, at each time instant, the next step is determined by the entire history of the past steps taken so far. Here, the walker starts from the origin and takes steps of unit size. At the first time instant, the walker moves to the right with probability $s\in[0,1]$ and to the left with probability $1-s$. At each time instant $n\ge 2$, the walker selects one of its previous steps uniformly at random and repeats it with probability $s'\in[0,1]$ or moves in the opposite direction with probability $1-s'$. That is, the first increment $\sigma_1$ has the following distribution:
\begin{equation*}
	\sigma_{1}=\begin{cases}
		+1\ \ \text{with probability $s$},\\
		-1\ \ \text{with probability $1-s$},
	\end{cases}
\end{equation*}
and for all $n\ge 1$, the subsequent increments are given by
\begin{equation*}
	\sigma_{n+1}=\begin{cases}
		+\sigma_{\beta(n)}\ \ \text{with probability $s'$},\\
		-\sigma_{\beta(n)}\ \ \text{with probability $1-s'$},
	\end{cases}
\end{equation*}
where $\beta(n)\sim$ unif$\{1,2,\dots,n\}$ which is independent of $\{\sigma_1,\sigma_2,\dots,\sigma_n\}$.
After $n\ge1$ steps, the position of the walker is $E_n=\sigma_1+\sigma_2+\dots+\sigma_n$.

Kumar {\it et al.} (2010) extended ERW in a way where the walker may choose to stay at rest. It is known as the ERW with stops. Baur and Bertoin (2016) have shown some connection between the ERW and certain P\'olya-type urn scheme. They established some functional limit theorems for the ERW. Coletti {\it et al.} (2017) studied asymptotic behaviour of the ERW in different regimes. Bercu (2018) established results on the law of large numbers (LLN), quadratic strong law (QSL), law of iterated logarithm (LIL) and asymptotic normality of the ERW using a martingale approach. Recently, Roy {\it et al.} (2025) studied the ERW in triangular array setting. For more details on ERW and its extensions, we refer the reader to  Miyazaki and Takei (2020), Gut and Stadtm\"uller (2021a, 2021b), Dedecker {\it et al.} (2023), Nakano (2025), Dhillon and Kataria (2026), and references therein. 

Bercu and Laulin (2019) introduced and studied the multidimensional elephant random walk (MERW). An extension of it, that is, the MERW with stops on $\mathbb{Z}^d$, $d\ge 1$ was studied by Bercu (2025). It is briefly described as follows:

Here, the walker starts from origin, and in the first step, the walker moves in any one of the $2d$ directions with equal probability $1/2d$. For $n\ge 1$, before making its $(n+1)$-th move, the walker chooses a step uniformly from the previous steps, that is, the steps till $n$-th time. Then, the walker repeats the chosen step with probability $p$ or moves towards any one of the remaining $2d-1$ directions with probability $q$ or chooses to stay at rest with probability $r$. That is,
for $n\ge 1$, the $(n+1)$-th step can be described as
\begin{equation}\label{multstep}
	X_{n+1}=A_{n+1}X_{\beta(n)},
\end{equation}
where $A_{n+1}$ and $\beta(n)$ are independent, and both are independent of $\{X_1,X_2,\dots,X_n\}$. Here, $\beta(n)\sim$ unif\{1,2,\dots,n\} and
\begin{align}\label{A_n+1_stops}
	A_{n+1}=\begin{cases*}
		+I_d &\text{with probability}\,\, $p$,\\
		-I_d &\text{with probability}\,\, $q$,\\
		+J_d &\text{with probability}\,\, $q$,\\
		-J_d &\text{with probability}\,\, $q$,\\
		+J_d^2 &\text{with probability}\,\, $q$,\\
		-J_d^2 &\text{with probability}\,\, $q$,\\
		\hspace{4mm}\vdots \\
		+J_d^{d-1} &\text{with probability}\,\, $q$,\\
		-J_d^{d-1} &\text{with probability}\,\, $q$,\\
		\hspace{4mm} O &\text{with probability}\,\, $r$,
	\end{cases*}
\end{align}
where
\begin{equation*}
	I_d=\begin{pmatrix}
		1 &0 &0 &\cdots &0 &0\\
		0 &1 &0 &\cdots &0 &0\\
		\vdots &\vdots &\vdots &\ddots &\vdots &\vdots\\
		0 &0 &0 &\cdots &1 &0\\
		0 &0 &0 &\cdots &0 &1
	\end{pmatrix}
	, \ 
	J_d=\begin{pmatrix}
		0 &1 &0 &\cdots &0 &0\\
		0 &0 &1 &\cdots &0 &0\\
		\vdots &\vdots &\vdots &\ddots &\vdots &\vdots\\
		0 &0 &0 &\cdots &0 &1\\
		1 &0 &0 &\cdots &0 &0
	\end{pmatrix}
\end{equation*}
and $O$ is the null matrix of order $d$.
Also,
\begin{equation}\label{pqr}
	p+q(2d-1)+r=1.
\end{equation}
Here, the $k$-th step $X_k$ is a $d$-dimensional column vector which takes values in the set $\{0, \pm e_1, \pm e_2,\dots, \pm e_d\}$, where $e_i$ denotes the $d$-dimensional column vector whose $i$-th entry is 1 and all other entries are zero. For example, for some $k$, if $X_k=-e_i$ then at the $k$-th step, the walker moves one unit in the direction of negative  $i$-th coordinate axis. For this model, Bercu (2025) obtained several convergence results including the LLN, LIL, and established asymptotic normality in different regimes. Recently, Qin (2025) proved a conjecture of Bertoin (2022) by establishing that the MERW on $\mathbb{Z}^d$ is transient for $d\ge 3$. It is shown that the model exhibits a phase transition in dimensions $d=1$ and $d=2$, separating the recurrent and transient regimes. The one-dimensional ERW with random step sizes is discussed in Fan and Shao (2024). Dedecker \textit{et al.} (2023) derived its LIL and proved the central limit theorem (CLT), and obtained the rates of convergence in CLT.

In this paper, first, we study the number of moves performed by the walker in MERW with stops. We derive the conditional mean increments of the number of moves and establish a recursive relation for its expectation. The asymptotic behaviour of the solution of this recursive relation  is discussed. A suitable multiplicative martingale is constructed using the conditional mean increments. By using this multiplicative martingale, we obtain several  convergence results for the number of moves that include LLN and LIL. Later, we study the MERW with random step sizes using a martingale approach. In this case, we discuss the number of moves and construct a suitable martingale for it. Using this martingale, we establish several almost sure convergence results including LLN, QSL, LIL and CLT for the number of moves in MERW with random step sizes. Also, we derive  similar convergence results for the walk. Our results complement those obtained by Bercu and Laulin (2019), Zhang (2024) and Bercu (2025).
\section{Preliminaries}
In this paper, we use the following notations: Let $\mathbb{C}$, $\mathbb{R}$, $\mathbb{Z}$ and $\mathbb{R}_+$ denote the sets of complex numbers, real numbers, integers and positive real numbers, respectively. Also, let $
\mathbb{Z}^d=\mathbb{Z}\times\mathbb{Z}\times\cdots\times\mathbb{Z}$ and $\mathbb{R}^d=\mathbb{R}\times\mathbb{R}\times\cdots\times\mathbb{R}$ ($d$ copies). For a set $G$, $\mathbb{I}_G$ denotes its indicator function, and $\mathfrak{R}(z)$ denotes the real part of a complex number $z$. For a real sequence $\{x_n\}_{n\ge 1}$, we denote $\lim_{n\to\infty} x_n$ by $x_\infty$. For two sequences $\{x_n\}_{n\ge 1}$ and $\{y_n\}_{n\ge 1}$, the notations $x_n \sim y_n$ and $x_n= o(y_n)$ indicate that $x_n/y_n \to 1$ and $x_n/y_n \to 0$, respectively, as $n \to \infty$. For any matrix $A$, its transpose is denoted by $A^t$, its conjugate transpose is denoted by $A^*$ and if $A$ is a square matrix, its trace is denoted by $\mathrm{tr}(A)$. Also, $\lambda_{\max}(A)$ and $\lambda_{\min}(A)$ denote the maximum and minimum eigenvalues of $A$, respectively. The abbreviation a.s.\ stands for almost surely. In expressions involving conditional expectations, this abbreviation is omitted to avoid repetition.
\subsection{Generalized hypergeometric function}\label{genhyper}
For $a_i\in \mathbb{C}$, $i=1,2,\dots,m$ and $b_j \in \mathbb{C}\backslash\{ 0,-1,-2,\dots\}$, $j=1,2,\dots, n$, the generalized hypergeometric function is defined as follows (see Kilbas \textit{et al.} (2006), Eq. (1.6.28)):
\begin{equation}
	\begin{aligned}\label{GenHyp}
		_mF_n(a_1,a_2,\dots,a_m; b_1,b_2,\dots,b_n;z)=\sum_{k=0}^{\infty}\frac{(a_1)_k(a_2)_k\dots(a_m)_k}{(b_1)_k(b_2)_k\dots(b_n)_k}\frac{z^k}{k!},\, z\in\mathbb{C},
	\end{aligned}
\end{equation}
where 
\begin{equation*}
	(x)_k \coloneqq\begin{cases*}
		x(x+1)\dots(x+k-1), \  k\ge1,\\
		1, \  k=0, 
	\end{cases*}
\end{equation*}
denotes the Pochhammer symbol. For $m\le n$, the series \eqref{GenHyp} is absolutely convergent for all $z\in \mathbb{C}$. For $m=n+1$, it is absolutely convergent if either $|z|=1$ with $\mathfrak{R}\big(\sum_{j=1}^{n}b_j-\sum_{i=1}^{m}a_i\big)>0$ or $|z|<1$.	

\subsection{$\mathbb{L}^p$ space}\label{Lmspac}
Let $(\Omega,\mathcal{F},\mathbb{P})$ be a probability space. Then, $\mathbb{L}^p$ space is the collection of all random variables $X$ on $(\Omega,\mathcal{F},\mathbb{P})$ whose $p$-th absolute moment is finite, that is, $	\mathbb{E}(|X|^p)<\infty$, $p\ge1$.
\subsection{Predictable square variation}\label{predictable}
Let $\{M_n\}_{n\ge 0}$ be a sequence of random vectors with values in $\mathbb{R}^d$ which is adapted to a filtration $\{\mathcal{F}_n\}_{n\ge 0}$. The predictable square variation of $\{M_n\}_{n\ge 0}$ is a random sequence $\{\langle M\rangle_n\}_{n\ge 0}$ of $d\times d$ positive semi-definite symmetric matrices defined as follows (see Duflo (1997), Definition 2.1.8):
\begin{equation*}
	\langle M\rangle_{n+1}-\langle M\rangle_n=\mathbb{E}((M_{n+1}-M_n)(M_{n+1}-M_n)^t|\mathcal{F}_n), \ n\ge 0,
\end{equation*}
with $\langle M\rangle_0=0$.	

\subsection{Law of large numbers} Here, we collect some law of large numbers for martingales.

\begin{theorem}[Duflo (1997), Theorem 1.3.15]\label{thm_LLN1}
	Let $\{M_n, \mathcal{F}_n\}_{n\ge 0}$ be a square-integrable martingale with predictable square variation $\{\langle M\rangle_n\}_{n\ge 0}$. Also, let $\langle M\rangle_\infty=\lim_{n\to\infty}\langle M\rangle_n$. Then,\vspace{.1cm}\\
	\noindent{(i)}  $\lim_{n\to\infty}M_n=M_\infty$ a.s. on $\{\langle M\rangle_\infty < \infty\}$, where $M_\infty$ is a finite random variable,\vspace{.1cm}\\		
	\noindent{(ii)}  $\lim_{n\to\infty}M_n/\langle M\rangle_n=0$ a.s. on $\{\langle M\rangle_\infty = \infty\}$.
\end{theorem}

\begin{theorem}[Duflo (1997), Theorem 1.3.24 ]\label{thm_LLN2}
	Let $\{\epsilon_n\}_{n\ge 0}$ be a square-integrable random sequence which takes values in $\mathbb{C}^d$ and is adapted to some filtration $\{\mathcal{F}_n\}_{n\ge 0}$ such that 
	\begin{equation*}
		\mathbb{E}(\epsilon_{n+1}|\mathcal{F}_n)=0 \,\,\,\text{and}\,\,\, \sup_{n\ge1}\mathbb{E}(\parallel\epsilon_{n+1}\parallel^2|\mathcal{F}_n)\le C, \ \text{a.s.},
	\end{equation*}
	 where $C$ is a finite random variable.
	Also, let $\{\Phi_n\}_{n\ge 0}$ be a sequence of complex $d$-dimensional random variables which is adapted to $\{\mathcal{F}_n\}_{n\ge 0}$, and $s_n=\sum_{k=0}^{n}\parallel\Phi_k\parallel^2$, $\lim_{n\to\infty}s_n=s_\infty$ and $M_n=\sum_{k=1}^{n}\langle \Phi_{k-1},\epsilon_k\rangle$. Then, \vspace{.1cm}\\
	\noindent{(i)}  $\{M_n\}_{n\ge 0}$ converges a.s. on the set $\{s_\infty <\infty\}$,	\vspace{.1cm}\\
	\noindent{(ii)} $\lim_{n\to\infty}M_n/s_{n-1}=0$ a.s. on the set $\{s_\infty =\infty\}$.
\end{theorem}
For the next result, we adopt the notations introduced in Section 4.2.1 of Duflo (1997).
\begin{theorem}[Duflo (1997), Theorem 4.3.15]\label{LLN}
Suppose that $\{M_n,\mathcal{F}_n\}_{n\ge 0}$ be a vector martingale with predictable square variation $\{\langle M\rangle_n\}_{n\ge0}$. For any continuous increasing function $f: \mathbb{R}_+ \to \mathbb{R}_+ $ such that $\int_{1}^{\infty}\mathrm{d}t/f(t) <\infty$. Then, on $\{s_\infty=\infty\}$, we have
\begin{equation*}
	\parallel \langle M\rangle_n^{-1/2} M_n\parallel^2=o\big(f(\log s_n)/\lambda_{\min}(\langle M\rangle_n)\big) \ \text{a.s.},
\end{equation*}
where $s_n=\text{tr}(\langle M\rangle_n)$.
\end{theorem}

\subsection{Quadratic strong law}
The following QSL for martingales will be used:
 \begin{theorem}[Bercu (2004), Theorem 3]\label{quadLaw}
	Let $\{\epsilon_n\}_{n\ge 1}$ be a sequence of real martingale differences adapted to an appropriate filtration $\{\mathcal{F}_n\}_{n\ge 1}$, and let $\{\Phi_n\}_{n\ge 0}$ be a sequence of random variables adapted to $\{\mathcal{F}_n\}_{n\ge 0}$. Define a real martingale transform $\{H_n\}_{n\ge 1}$ by $H_n=\sum_{k=1}^{n}\Phi_{k-1}\epsilon_k$
	and the explosion coefficient associated with $\{\Phi_n\}_{n\ge 0}$ by $f_n=\Phi_n^2/v_n$, where $v_n=\sum_{k=0}^{n}\Phi_k^2$.
	
Let $\{\epsilon_n\}_{n\ge 1}$ satisfies the following conditions:\vspace{0.1cm}\\
\noindent{(i)} $\mathbb{E}(\epsilon_{n+1}^2|\mathcal{F}_n)=\sigma^2$ a.s.,\vspace{0.1cm}\\
\noindent{(ii)} for some integer $m\ge 1$, $\sup_{n\ge 0}\mathbb{E}(|\epsilon_{n+1}|^a|\mathcal{F}_n)<\infty$, for some $a> 2m$,\vspace{0.1cm}\\
\noindent{(iii)} the explosion coefficient $f_n$ tends to zero a.s.\vspace{0.1cm}\\
Then,
	\begin{equation*}
		\lim_{n\to \infty}\frac{1}{\log v_n}\sum_{k=1}^{n}f_k \bigg(\frac{H_k^2}{v_{k-1}}\bigg)^m = \frac{(2m)!\sigma^{2m}}{m! 2^m} \ \text{a.s.}
	\end{equation*}
\end{theorem}

\begin{proposition}[Duflo (1997), Proposition 4.2.8]\label{quadform}
	Let $\mathcal{H}$ denote the set of positive semi-definite Hermitian matrices of order $d$. Suppose $\{A_n\}_{n\ge 1}$ and $A$ are random matrices with values in $\mathcal{H}$ and associated quadratic forms are $q_{A_n}$ and $q_A$, respectively, such that $
	\lim_{n\to \infty}q_{A_n}(u)=q_A(u)$ for all $u\in\mathbb{C}^d$. Here, for a matrix $A$, the associated quadratic form is defined as $q_A(u)=u^*Au$.
	Then, $\lim_{n\to \infty}A_n=A$ a.s.	
\end{proposition}
\subsection{Law of iterated logarithm}
The following iterated logarithm result for a martingale holds (see Duflo (1997), Corollary 6.4.25):
\begin{theorem}\label{LIL}
Let $\{\epsilon_n\}_{n\ge 0}$ be a sequence of random variables adapted to filtration $\{\mathcal{F}_n\}_{n\ge 0}$ such that a.s.\vspace{.1cm}\\
	\noindent{(i)} $\mathbb{E}(\epsilon_{n+1}|\mathcal{F}_n)=0$ and $\mathbb{E}(\epsilon^2_{n+1}|\mathcal{F}_n)\le \sigma^2$, $n\ge 1$,\vspace{.1cm}\\
	\noindent{(ii)} for some $0<\alpha<1$, we have  $\displaystyle \sup_{n\ge1} \mathbb{E}(|\epsilon_{n+1}|^{2+2\alpha}|\mathcal{F}_n)<\infty$.\\
	Also, let $\{\Phi_n\}_{n\ge 0}$ and $\{U_n\}_{n\ge 0}$ be adapted to $\{\mathcal{F}_n\}_{n\ge 0}$ such that $|\Phi_n|\le U_n$, and $M_n=\sum_{k=1}^{n}\Phi_{k-1}\epsilon_k$ and $\tau_n=\sum_{k=0}^{n}U_k^2$. If $
	\tau_\infty=\infty$ and $\sum_{n=0}^{\infty}U_n^{2+2\alpha}\tau_n^{-1-\alpha}<\infty$ a.s. then
\begin{equation*}
\limsup_{n\to\infty}\frac{|M_n|}{\sqrt{2\tau_{n-1}\log\log \tau_{n-1}}}\le \sigma \ \text{a.s.}
\end{equation*}
\end{theorem}
\subsection{Central limit theorem}
The following CLT for martingales will be used:	
\begin{theorem}[Duflo (1997), Corollary 2.1.10]\label{CLT}	Let $\{M_n,\mathcal{F}_n\}_{n\ge 0}$ be a real, square-integrable vector martingale whose predictable square variation is $\{\langle M\rangle_n\}_{n\ge 0}$. Suppose that for a real, deterministic sequence $\{b_n\}_{n\ge 0}$ increasing to $\infty$, the following assumptions hold:\\
\noindent {(i)} $b_n^{-1}\langle M\rangle_n\xrightarrow{p}\Lambda$, where $\Lambda$ is a deterministic symmetric positive semi-definite matrix,\\
\noindent {(ii)} Lindeberg's condition is satisfied, that is, for all $\epsilon>0$,
\begin{equation*}
b_n^{-1}\sum_{k=1}^{n}\mathbb{E}(\parallel M_k-M_{k-1}\parallel^2 \mathbb{I}_{\{\parallel M_k-M_{k-1}\parallel \ge \epsilon \sqrt{b_n}\}}|\mathcal{F}_{k-1})\xrightarrow{p} 0.
\end{equation*}
Then, $b_n^{-1/2}M_n\xrightarrow{d} \mathcal{N}(0,\Lambda)$ and     $\lim_{n\to\infty}b_n^{-1}M_n=0
$ a.s. Here, $\xrightarrow{p} $ and $\xrightarrow{d}$ denotes the convergence in probability and convergence in distribution, respectively. 
\end{theorem}
\section{MERW with stops}	
Here, we study the number of moves in MERW with stops. The steps at which the walker remains at rest are referred as delays. Note that the number of moves in MERW with stops till $n$-th step is at least $1$ which follows from its construction. Thus, the number of delays till $n$-th step is at most $n-1$. 

Let $Z_n^*$ and $Z_n$ denote the number of moves and delays till $n$-th step, respectively. Then, for all $n\ge 1$, we have
\begin{equation}\label{0+1}
	Z_n^* + Z_n=n.
\end{equation}
Also, 
\begin{equation}\label{ones}
	Z_n^*=\sum_{k=1}^{n}X_k^t X_k= \sum_{k=1}^{n}\parallel X_k\parallel^2,\ n\ge 1.
\end{equation}
Let $\mathcal{G}_n=\sigma(X_1,X_2,\dots,X_n)$ be the $\sigma$-field generated by the complete past till $n$-th step, and $G\in \mathcal{G}_n$. Then, 
\begin{align}
	\int_G \mathbb{E}(X_{n+1}^t X_{n+1}|\mathcal{G}_n)\,\mathrm{d}\mathbb{P}
	&=\int_G X_{n+1}^t X_{n+1}\,\mathrm{d}\mathbb{P} \nonumber \\
	&=\mathbb{E}(X_{\beta(n)}^t A_{n+1}^t A_{n+1}X_{\beta(n)}\mathbb{I}_G), \ \text{(using \eqref{multstep})} \nonumber\\
	&=\sum_{k=1}^{n}\mathbb{P}\{\beta(n)=k\}\mathbb{E}(X_{\beta(n)}^t A_{n+1}^t A_{n+1}X_{\beta(n)}\mathbb{I}_G | \beta(n)=k) \nonumber\\
	&=\frac{1}{n}\sum_{k=1}^{n}\mathbb{E}(X_k^t A_{n+1}^t A_{n+1}X_k \mathbb{I}_G) \label{indep01}\\
	&=\frac{1}{n}\sum_{k=1}^{n}\mathbb{P}\{A_{n+1}=I_d\}\mathbb{E}(X_k^t A_{n+1}^t A_{n+1}X_k \mathbb{I}_G|A_{n+1}=I_d)\nonumber\\
	&\ \ +\frac{1}{n}\sum_{k=1}^{n}\mathbb{P}\{A_{n+1}=-I_d\}\mathbb{E}(X_k^t A_{n+1}^t A_{n+1}X_k \mathbb{I}_G|A_{n+1}=-I_d)\nonumber\\
	&\ \ +\frac{1}{n}\sum_{k=1}^{n}\sum_{i=1}^{d-1}\mathbb{P}\{A_{n+1}=J_d^i\}\mathbb{E}(X_k^t A_{n+1}^t A_{n+1}X_k \mathbb{I}_G| A_{n+1}=J_d^i)\nonumber\\
	&\ \ +\frac{1}{n}\sum_{k=1}^{n}\sum_{i=1}^{d-1}\mathbb{P}\{A_{n+1}=-J_d^i\}\mathbb{E}(X_k^t A_{n+1}^t A_{n+1}X_k \mathbb{I}_G|A_{n+1}= -J_d^i)\nonumber\\
	&=\frac{p+q}{n}\sum_{k=1}^{n} \mathbb{E}(X_k^t X_k \mathbb{I}_G) +\frac{2q}{n}\sum_{k=1}^{n}\sum_{i=1}^{d-1}\mathbb{E}(X_k^t X_k \mathbb{I}_G)\nonumber\\
	&=\frac{p+q(2d-1)}{n}\sum_{k=1}^{n}\int_G \parallel X_k \parallel^2 \,\mathrm{d}\mathbb{P}, \nonumber
\end{align}
where \eqref{indep01} follows from the independence of $\beta(n)$ with $\{X_1, X_2,\dots,X_n\}$ and $A_{n+1}$. Also, the penultimate step follows from the independence of $A_{n+1}$ and $\{X_1, X_2,\dots,X_n\}$.

As $X_k$, $k=1,2,\dots, n$ are $\mathcal{G}_n$-measurable,  by using Radon-Nikodym theorem, we get
\begin{equation}\label{norm_n+1}
	\mathbb{E}(\parallel X_{n+1} \parallel^2 |\mathcal{G}_n)=\frac{1-r}{n}Z_n^*,\ n\ge 1,
\end{equation}	
which follows from \eqref{pqr} and \eqref{ones}. Again, from \eqref{ones} and \eqref{norm_n+1}, we have
\begin{equation}\label{n+1_star}
	\mathbb{E}(Z_{n+1}^*|\mathcal{G}_n)=\Big(1+\frac{1-r}{n}\Big)Z_n^*, \ n\ge 1.
\end{equation}
Thus,
\begin{equation}\label{star1}
	\mathbb{E}(Z_{n+1}^*)=\Big(1+\frac{1-r}{n}\Big)\mathbb{E}(Z_n^*), \ n\ge 1.
\end{equation}
On solving the recurrence relation \eqref{star1}, we obtain
\begin{equation*}
	\mathbb{E}(Z_n^*)=\frac{\Gamma(n+1-r)}{\Gamma(2-r)\Gamma(n)}.
\end{equation*}
Now, by using Eq. (11) of Sch\"utz and Trimper (2004), the asymptotic behaviour of $\mathbb{E}(Z_n^*)$ is obtained in the following form:
\begin{equation}\label{star2}
	\mathbb{E}(Z_n^*)\sim \frac{n^{1-r}}{\Gamma(2-r)} \ \text{as} \ n\to\infty.
\end{equation}	
\begin{remark}
	For $r<1$ in \eqref{star2}, it follows that the expected number of moves approaches $\infty$ as $n \to \infty$. That is, in a longer run the walker performs infinitely many moves on an average in this regime.
\end{remark}
From  \eqref{n+1_star}  
and Lemma 2.1 of Gut and Stadtm\"uller (2021b), we have the following lemma:
\begin{lemma}\label{lemmultmar}
	Let $a_1=1$ and
	\begin{equation}\label{a_k}
		a_n=\prod_{k=1}^{n-1}\Big(1+ \frac{1-r}{k}\Big),\ n\ge 2.
	\end{equation}
	Then, $\{\mathcal{M}_n,\mathcal{G}_n\}_{n\ge 1}$ is a martingale, where $\mathcal{M}_n=Z_n^*/a_n$ and $\mathcal{G}_n=\sigma(X_1,X_2,\dots,X_n)$.
\end{lemma}
Bercu (2018) referred the process $\{\mathcal{M}_n,\mathcal{G}_n\}_{n\ge 1}$ as   multiplicative martingale.
By using the Stirling's approximation for  gamma function (see Kilbas \textit{et al.} (2006), p. 25), we get
\begin{equation}\label{asym_a_n}
	a_n\sim \frac{n^{1-r}}{\Gamma(2-r)} \ \text{as} \ n\to\infty.
\end{equation}

To obtain the LLN and LIL-type results for $Z_n^*$,
we consider the following martingale differences:
\begin{equation}\label{M_nDifference1}
	\Delta \mathcal{M}_n= \mathcal{M}_n-\mathcal{M}_{n-1},\,\, n\ge 1,
\end{equation}
with $\mathcal{M}_0=0$. Equivalently, from Lemma \ref{lemmultmar}, we have
\begin{equation}\label{delta_Mn1}
	\Delta \mathcal{M}_n=\frac{\delta_n}{a_n},
\end{equation}
where $\delta_1=1$ and $\delta_n=Z_n^*-\big(1+\frac{1-r}{n-1}\big)Z_{n-1}^*$, $n\ge2$. 

From \eqref{M_nDifference1} and \eqref{delta_Mn1}, we have
\begin{equation}\label{additive1}
	\mathcal{M}_n=\sum_{k=1}^{n}\frac{\delta_k}{a_k}.
\end{equation}

Now, by using \eqref{M_nDifference1} and \eqref{delta_Mn1}, we obtain
\begin{align}	
	\mathbb{E}(\delta_{n+1}^2|\mathcal{G}_n)&=a_{n+1}^2\big(\mathbb{E}(\mathcal{M}_{n+1}^2|\mathcal{G}_n)-\mathcal{M}_n^2\big)\nonumber\\
	&=\mathbb{E}(Z_{n+1}^{*2}|\mathcal{G}_n)-\Big(1+\frac{1-r}{n}\Big)^2 Z_n^{*2}\nonumber\\
	&= Z_n^{*2}+2Z_n^*\mathbb{E}(\parallel X_{n+1}\parallel^2|\mathcal{G}_n)+\mathbb{E}(\parallel X_{n+1}\parallel^4|\mathcal{G}_n)-\Big(1+\frac{1-r}{n}\Big)^2 Z_n^{*2}\nonumber\\
	&\le Z_n^{*2}+\frac{2(1-r)}{n}Z_n^{*2}+1-\Big(1+\frac{1-r}{n}\Big)^2 Z_n^{*2},\ \text{(using \eqref{norm_n+1})}\nonumber\\
	&=1-\Big(\frac{1-r}{n}Z_n^*\Big)^2 \nonumber\\
	& \le 1\label{delta^2}.
\end{align}
Also, from \eqref{delta_Mn1}, we have
\begin{equation}\label{delta_n+1=0}
	\mathbb{E}(\delta_{n+1}|\mathcal{G}_n)=0
\end{equation}
and from \eqref{delta^2}, it follows that
\begin{equation}\label{suple1}
	\sup_{n\ge1} \mathbb{E}(\delta_{n+1}^2|\mathcal{G}_n)\le 1.
\end{equation}

Let us fix the following notations:
\begin{equation}\label{def_Phi_sn}
	\Phi_{k-1}=\frac{1}{a_k} \ \ \text{and} \ \ u_n=\sum_{k=0}^{n}\Phi_k^2.
\end{equation}
Thus, from \eqref{a_k}, we have
\begin{equation*}
	u_{n-1}=\sum_{k=1}^{n}1/a_k^2=\sum_{k=1}^{n}\bigg(\frac{\Gamma(2-r)\Gamma(k)}{\Gamma(k+1-r)}\bigg)^2.
\end{equation*}

The proof of the next result follows on using the Stirling's formula for gamma function (see Kilbas \textit{et al.} (2006), p. 25).
\begin{proposition}\label{prplimres}
Let $\{u_n\}_{n\ge 0}$ be as defined in \eqref{def_Phi_sn}. Then, we have\vspace{.1cm}\\
\noindent{(i)} $\lim_{n\to\infty}\frac{u_{n-1}}{n^{2r-1}}=\frac{(\Gamma(2-r))^2}{2r-1}$, $1/2<r\le 1$,\vspace{.2cm}\\ 
\noindent{(ii)} $\lim_{n\to\infty}\frac{u_{n-1}}{\log n}=\frac{\pi}{4}$,  $r=1/2$ and\vspace{.2cm} \\
\noindent{(iii)} $\lim_{n\to\infty}u_{n-1}= {_3F_2(1,1,1;2-r,2-r;1)}$, $0\le r<1/2$,\vspace{.1cm}\\
where ${_3F_2(\cdot)}$ is the generalized hypergeometric function defined in \eqref{genhyper}.
\end{proposition}

\subsection{Law of large numbers}
Here, we state and prove some LLN-type results for $Z_n^*$.
\begin{theorem}\label{lln_res1}
	Let $1/2<r \le 1$. Then, $\lim_{n\to\infty}Z_n^*/n^r=0\ \text{a.s.}$
\end{theorem}
\begin{proof}
	By using \eqref{delta_n+1=0} and \eqref{suple1}, the first two conditions of Theorem \ref{thm_LLN2} are satisfied. Also, from Proposition \ref{prplimres}(i), we have $u_\infty=\infty$.
	By using \eqref{additive1} and \eqref{def_Phi_sn}, and applying Theorem \ref{thm_LLN2}, we get
	\begin{equation}\label{Mn_sn1}
		\lim_{n\to\infty}\frac{\mathcal{M}_n}{u_{n-1}}=0 \ \text{a.s.},
	\end{equation}
	where $\mathcal{M}_n=Z_n^*/a_n$. From \eqref{asym_a_n} and Proposition \ref{prplimres}(i), we have 
	\begin{equation*}
		\frac{1}{a_n u_{n-1}}\sim \frac{2r-1}{\Gamma(2-r)n^r} \ \text{as} \ n\to\infty.
	\end{equation*}
	Finally, from \eqref{Mn_sn1}, it follows that
	\begin{equation*}
		\lim_{n\to\infty}\frac{Z_n^*}{n^r}\frac{2r-1}{\Gamma(2-r)}=0 \ \text{a.s.},
	\end{equation*}
	which reduces to the required result.	
\end{proof}
\begin{theorem}\label{lln_res2}
	Let $r=1/2$. Then, $\lim_{n\to\infty}\frac{Z_n^*}{\sqrt{n} \log n}=0 \ \text{a.s.}$
\end{theorem}
\begin{proof}
	By using \eqref{delta_n+1=0} and \eqref{suple1}, the first two conditions of Theorem \ref{thm_LLN2} are satisfied. From Proposition \ref{prplimres}(ii), we have $u_\infty=\infty$.
	By using \eqref{additive1}, \eqref{def_Phi_sn} and Theorem \ref{thm_LLN2}, we get
	\begin{equation}\label{Mn_sn2}
		\lim_{n\to\infty}\frac{\mathcal{M}_n}{u_{n-1}}=0 \ \text{a.s.}
	\end{equation}
	Also, by using \eqref{asym_a_n} and Proposition \ref{prplimres}(ii), we obtain
	\begin{equation*}
		\frac{1}{a_n u_{n-1}}\sim \frac{2}{ \sqrt{n\pi}\log n} \ \text{as} \ n\to\infty.
	\end{equation*}
	Thus, 
	\begin{equation*}
		\lim_{n\to\infty}\frac{Z_n^*}{\sqrt{n}\log n}=0 \ \text{a.s.},
	\end{equation*}
	which follows from \eqref{Mn_sn2}.	This completes the proof.	
\end{proof}

The following lemma (see Bercu (2022)) will be used to prove the next result:
\begin{lemma}\label{Lm bdd}
	Let $0\le r\le1$ and $m\ge 1$ be an integer. Then, the martingale $\{\mathcal{M}_n,\mathcal{G}_n\}_{n\ge 1}$ is $\mathbb{L}^m$-bounded, that is, $\sup_{n\ge 1} \mathbb{E}(\mathcal{M}_n^m) < \infty$, where $\mathbb{L}^m$ is defined in Section \ref{Lmspac}.
	Thus, $\{\mathcal{M}_n\}_{n\ge 1}$ converges a.s. as well as in $\mathbb{L}^m$ to a finite random variable $\mathcal{M}$ which satisfies
	\begin{equation*}
		\mathbb{E}(\mathcal{M}^m)=\frac{m! (\Gamma(2-r))^m}{\Gamma(m-mr+1)}.
	\end{equation*}
\end{lemma}
\begin{theorem}\label{lln_res3}
	For $0\le r<1/2$, we have
	\begin{equation}\label{conv to Z}
		\lim_{n\to\infty}Z_n^*/n^{1-r}=Z\ \text{a.s.},
	\end{equation}
	where $Z$ is a finite random variable. Moreover, this convergence holds in $\mathbb{L}^m$ for any integer $m\ge 1$, that is,
	\begin{equation}\label{Lm conv}
		\lim_{n\to\infty}\mathbb{E}\bigg(\bigg |\frac{Z_n^*}{n^{1-r}}-Z\bigg |^m \bigg)=0  \ \text{a.s.}
	\end{equation}
\end{theorem}
\begin{proof}
	The predictable square variation of $\{\mathcal{M}_n\}_{n\ge 1}$ is given by (see Section \ref{predictable})
	\begin{equation*}
		\langle \mathcal{M}\rangle_n
		=\sum_{k=1}^{n}\mathbb{E}((\Delta \mathcal{M}_k)^2|\mathcal{G}_{k-1}),
	\end{equation*}
	where $\Delta \mathcal{M}_k$ is defined in \eqref{M_nDifference1}.\\
	As $\mathcal{M}_k=Z_k^*/a_k$, by using \eqref{delta_Mn1}, we have
	\begin{equation}\label{delMn_1st}
		\langle \mathcal{M}\rangle_n =\sum_{k=1}^{n}\frac{1}{a_k^2}\mathbb{E}(\delta_k^2|\mathcal{G}_{k-1}).
	\end{equation}
	Also, by using \eqref{delta^2} and  \eqref{def_Phi_sn} in \eqref{delMn_1st}, we get $\langle \mathcal{M}\rangle_n \le u_{n-1}$. From Proposition \ref{prplimres}(iii), it follows that $\{u_n\}_{n\ge 0}$ converges to a finite limit. Also, $\{\langle \mathcal{M}\rangle_n\}_{n\ge 1}$ is an increasing sequence which is bounded above by the convergent sequence $\{u_n\}_{n\ge 0}$. Thus, $\{\langle \mathcal{M}\rangle_n\}_{n\ge 1}$ is also convergent. By using Theorem \ref{thm_LLN1}, we have 
	\begin{equation}\label{M_inf}
		\lim_{n\to\infty}\mathcal{M}_n=\mathcal{M}_\infty\  \text{a.s.},
	\end{equation}
	where $\mathcal{M}_\infty=\sum_{k=1}^{\infty}\delta_k/a_k$, which is a finite random variable. Moreover, the convergence in \eqref{conv to Z} follows from \eqref{asym_a_n} and \eqref{M_inf}. The convergence in \eqref{Lm conv} then follows as a consequence of Lemma \ref{Lm bdd}. This completes the proof.
	\end{proof}
\begin{remark} 
For an alternate proof of Theorem \ref{lln_res3}, we refer the reader to Lemma 2.1 of Bercu (2025). We note that Lemma 2.1 of Bercu (2025) generalizes Theorem \ref{lln_res3} for interval $(0,1)$, and specifies that the random variable $Z$ has Mittag-Leffler distribution. 
\end{remark}
\begin{corollary}
From Theorem \ref{lln_res1}, Theorem \ref{lln_res2} and Theorem \ref{lln_res3}, we have the following convergence results for the number of delays $Z_n$:\vspace{.1cm}\\
\noindent{(i)} $\lim_{n\to\infty}\Big(\frac{Z_n}{n^r}- n^{1-r}\Big)=0  \ \text{a.s.}, \ 1/2<r\le 1$,\vspace{.1cm}\\
\noindent{(ii)} $\lim_{n\to\infty}\Big(\frac{Z_n}{\sqrt{n}\log n}- \frac{\sqrt{n}}{\log n}\Big)=0  \ \text{a.s.}, \ r=1/2$,\vspace{.1cm}\\
\noindent{(iii)} $\lim_{n\to\infty}\Big(\frac{Z_n}{n^{1-r}}- n^r \Big)=\tilde{Z}  \ \text{a.s.}, \ 0\le r<1/2$,\\
respectively, where $\tilde{Z}$ is a finite random variable, and the above limits follow from \eqref{0+1}.
\end{corollary}

\subsection{Law of iterated logarithm}
Here, we establish some LIL-type results for the number of moves in MERW with stops.
\begin{lemma}\label{LILlem1}
	Let $\delta_n=Z_n^*-\big(1+\frac{1-r}{n-1}\big)Z_{n-1}^*$, $n\ge2$. Then, 
	$\displaystyle \sup_{n\ge1}\mathbb{E}(|\delta_{n+1}|^3|\mathcal{G}_n)<\infty$.
\end{lemma}

\begin{proof}
	By using \eqref{ones}, we have
	\begin{equation*}
		\delta_{n+1}=Z_{n+1}^*-\Big(1+\frac{1-r}{n}\Big)Z_{n-1}^*=\parallel X_{n+1}\parallel ^2 - \frac{1-r}{n}Z_n^*.
	\end{equation*}
	That is,
	\begin{equation*}
		|\delta_{n+1}| \le \parallel X_{n+1}\parallel ^2 + \frac{1-r}{n}Z_n^* \le 2-r.
	\end{equation*}
	Hence, $|\delta_{n+1}|^3 \le 8$. This completes the proof.
\end{proof}
\begin{theorem}\label{LIL1}
	For $1/2<r\le 1$, we have 
	\begin{equation*}
		\limsup_{n\to\infty}\frac{Z_n^*}{\sqrt{2n\log\log n}} \le \frac{1}{\sqrt{2r-1}} \ \text{a.s.}
	\end{equation*}
\end{theorem}
\begin{proof}
	For $\alpha=1/2$, in view of \eqref{delta^2}, \eqref{delta_n+1=0} and Lemma \ref{LILlem1}, the first two conditions of Theorem \ref{LIL} are satisfied.
	
	Consider $U_n$ of Theorem \ref{LIL} as $\Phi_n$. Then, 
	\begin{equation*}
		\tau_n=\sum_{k=0}^{n}U_k^2=\sum_{k=0}^{n}\Phi_k^2=u_n.
	\end{equation*}
	By using Proposition \ref{prplimres}(i), we get $u_\infty=\infty$.	
	From \eqref{asym_a_n} and  Proposition \ref{prplimres}(i), we have
	\begin{align}
		\Phi_n&=\frac{1}{a_{n+1}}\sim \frac{\Gamma(2-r)}{(n+1)^{1-r}}, \label{asym1}\\
		u_n &\sim \frac{(\Gamma(2-r))^2}{2r-1}(n+1)^{2r-1}, \label{asym2}
	\end{align}
	respectively, as $n\to\infty$.
	By using \eqref{asym1} and \eqref{asym2}, we get
	\begin{equation*}
		\Phi_n^3 u_n^{-3/2}\sim \bigg(\frac{2r-1}{n+1}\bigg)^{3/2} \ \text{as} \ n\to\infty.
	\end{equation*}
	Thus, $\sum_{n=0}^{\infty}\Phi_n^3 u_n^{-3/2} < \infty \ \text{a.s.}$
	As all the conditions of Theorem \ref{LIL} are satisfied, it follows that
	\begin{equation}\label{1st}
		\limsup_{n\to\infty}\frac{|\mathcal{M}_n|}{\sqrt{2 u_{n-1}\log\log u_{n-1}}}\le 1  \ \text{a.s.}
	\end{equation}
	Recall that $\mathcal{M}_n=Z_n^*/a_n$ and $\frac{1}{a_n^2 u_{n-1}}\sim\frac{2r-1}{n}$. From \eqref{1st}, we have
	\begin{equation*}
		\limsup_{n\to\infty}\frac{Z_n^*}{\sqrt{2n\log\log u_{n-1}}}\le \frac{1}{\sqrt{2r-1}} \ \text{a.s.}
	\end{equation*}
	As $n\to\infty$, we have $\log\log u_{n-1}\sim \log\log n$, which yields the required result.		
\end{proof}
The proof of next result follows along the similar lines to that of Theorem \ref{LIL1}.
\begin{theorem}
	For $r=1/2$, we have 
	\begin{equation*}
		\limsup_{n\to\infty}\frac{Z_n^*}{\sqrt{2n\log n\log\log\log n}}\le 1  \ \text{a.s.}
	\end{equation*}
\end{theorem}
\begin{proof}
	For $\alpha=1/2$, the first two conditions of Theorem \ref{LIL} follow from \eqref{delta^2}, \eqref{delta_n+1=0} and Lemma \ref{LILlem1}.	
	
	From Proposition \ref{prplimres}(ii), we have $u_\infty=\infty$ and
	\begin{equation}\label{asym3}
		u_n \sim \frac{\pi}{4}\log(n+1) \ \text{as} \ n\to\infty.
	\end{equation}
	By using \eqref{asym1} in \eqref{asym3}, we obtain
	\begin{equation*}
		\Phi_n^3 u_n^{-3/2}\sim \frac{1}{((n+1)\log(n+1))^{3/2}} \ \text{as} \ n\to\infty.
	\end{equation*}
	Consequently, for all $n\ge 2$, we have
	\begin{equation*}
		\Phi_n^3 u_n^{-3/2} \le \frac{1}{(n+1)^{3/2}},
	\end{equation*}
	which gives
	\begin{equation*}
		\sum_{n=0}^{\infty}\Phi_n^3 u_n^{-3/2} < \infty \ \text{a.s.}
	\end{equation*}
As all the conditions of Theorem \ref{LIL} are satisfied, it follows that
	\begin{equation}\label{2nd}
	\limsup_{n\to\infty}\frac{|\mathcal{M}_n|}{\sqrt{2 u_{n-1}\log\log u_{n-1}}}\le 1  \ \text{a.s.}
	\end{equation}
	Recall that $\mathcal{M}_n=Z_n^*/a_n$ and  $\frac{1}{a_n^2 u_{n-1}}\sim\frac{1}{n^{2-2r}\log n}=\frac{1}{n\log n}$. From \eqref{2nd}, we get
	\begin{equation}\label{3rd}
	\limsup_{n\to\infty}\frac{Z_n^*}{\sqrt{2n\log n\log\log u_{n-1}}}\le 1 \ \text{a.s.}
	\end{equation}
Also, from Proposition \ref{prplimres}(ii), we have
	\begin{equation}\label{4th}
		\log\log u_{n-1}\sim \log\log\log n \ \ \text{as} \ n\to\infty.
	\end{equation}
	Finally, by using \eqref{4th} in \eqref{3rd}, we get the required result.	
\end{proof}

\section{MERW with random step sizes}
Zhang (2024) obtained a Gaussian approximation for the multidimensional elephant random walk with random step sizes on $\mathbb{R}^d$, $d\ge 1$, using the theory of the recursive stochastic algorithm. From this Gaussian approximation, several asymptotic results including central limit theorem, law of the iterated logarithm of multidimensional elephant random walk with random step sizes are derived. Here, we do a similar study using the martingale approach of Bercu and Laulin (2019). 

In MERW with stops, the walker is allowed to remain at rest for all except the first step. On substituting $r=0$ in \eqref{A_n+1_stops}, the MERW with stops reduces to the MERW. Henceforth, $A_{n+1}$, $n\ge 1$ take values from $\{I_d\}$ and $\{-I_d,\pm J_d,\pm J_d^2,\dots,\pm J_d^{d-1}\}$ with probability $p$ and $(1-p)/(2d-1)$, respectively. Thus, the position of the walker after $n$ steps in MERW is given by
\begin{equation}\label{W_n}
W_n=\sum_{k=1}^{n}X_k, \ n\ge 1.
\end{equation}

The MERW with random step sizes  $\{S_n\}_{n\ge 1}$ is defined as follows:
\begin{equation}\label{S_n}
	S_n\coloneqq\sum_{k=1}^{n}T_k, \ n\ge 1.
\end{equation}
In this model, for $n=1$, the walker can move in any of the $2d$ possible directions with probability $1/2d$. The first step is defined by
\begin{equation}\label{first-step}
T_1\coloneq X_1 Y_1,
\end{equation}
where the random variable $X_1$ takes values in $\{\pm e_1, \pm e_2,\dots, \pm e_d\}$ with equal probability $1/2d$. Also, the size of first step is determined by the positive random variable $Y_1$ having finite mean $\mu_1$ and finite variance $\eta_1^2$. For all $n\ge 1$, the $(n+1)$-th step is determined by
\begin{equation*}
	T_{n+1}\coloneq A_{n+1} X_{\beta(n)}Y_{n+1},
\end{equation*}
where $A_{n+1}$, $\beta(n)$ and $Y_{n+1}$ are independent, and they are independent of $\{X_1, X_2,\dots, X_n\}$ and $Y_1$. Here, each $X_k$ considers values in the set $\{\pm e_1,\pm e_2,\dots,\pm e_d\}$ with equal probability $1/2d$.  
The step sizes are determined by the sequence $\{Y_n\}_{n\ge 2}$ of non-negative, independent and identically distributed (iid) random variables with finite mean $\mu $ and finite variance $\eta^2$. We assume that the walker starts from the origin, that is, the initial position $S_0=0$, and the position of the walker after $n\ge1$ steps is $S_n$.

Let $\mathcal{F}_n=\sigma(T_1,T_2,\dots,T_n)$, $n\ge1$ be the $\sigma$-field generated by the entire history of the walk up to step $n$.
For any $G\in \mathcal{F}_n$, we have
\begin{align*}
	\int_G \mathbb{E}(A_{n+1}X_{\beta(n)}(Y_{n+1}-\mu)|\mathcal{F}_n)\,\mathrm{d}\mathbb{P}
	&=\int_G A_{n+1}X_{\beta(n)}(Y_{n+1}-\mu)\,\mathrm{d}\mathbb{P}\\
	&=\mathbb{E}(A_{n+1}X_{\beta(n)}(Y_{n+1}-\mu)\mathbb{I}_G)\\
	&=\mathbb{E}(Y_{n+1}-\mu)\mathbb{E}(A_{n+1}X_{\beta(n)}\mathbb{I}_G)\\
	&=0,
\end{align*}
where the penultimate step follows from the independence of $Y_{n+1}$ with $A_{n+1}$, $\beta(n)$ and $\{X_1,X_2,\dots,X_n\}$. The last step follows from the assumption that $\{Y_n\}_{n\ge 2}$ is a sequence of iid random variables with finite mean $\mu $. Finally, by using Radon-Nikodym theorem, we get
\begin{equation}\label{Y-mu}
	\mathbb{E}(A_{n+1}X_{\beta(n)}(Y_{n+1}-\mu)|\mathcal{F}_n)=0, \ n\ge 1.
\end{equation}

\begin{lemma}
For $n\ge 1$, let $W_n$ and $S_n$ be as given in \eqref{W_n} and \eqref{S_n}, respectively. Also, let $M_n=S_n- \mu W_n$ and $\mathcal{F}_n=\sigma(T_1,T_2,\dots,T_n)$. Then, $\{M_n,\mathcal{F}_n\}_{n\ge 1}$ is a martingale.
\end{lemma}
\begin{proof}
By using \eqref{W_n} and \eqref{S_n}, we have
	\begin{align*}
		\mathbb{E}(M_{n+1}|\mathcal{F}_n)
		&=\mathbb{E}(S_n+T_{n+1}-\mu W_n-\mu X_{n+1}|\mathcal{F}_n)\\
		&=M_n+\mathbb{E}(A_{n+1}X_{\beta(n)}(Y_{n+1}-\mu)|\mathcal{F}_n)\\
		&=M_n,
	\end{align*}
	where the last step follows from \eqref{Y-mu}.
	This completes the proof.
\end{proof}

Consider the following martingale differences:
\begin{equation}\label{Delta_Mn}
	\Delta M_n= M_n - M_{n-1}, \ n\ge 1,
\end{equation}
with $M_0=0$.

For martingale $\{M_n,\mathcal{F}_n\}_{n\ge 1}$, its predictable square variation $\langle M\rangle_n$, $n\ge0$ defined in Section \ref{predictable} can be equivalently written as
\begin{equation}
	\langle M\rangle_n =\sum_{k=1}^{n}\mathbb{E}(\Delta M_k (\Delta M_k)^t|\mathcal{F}_{k-1}), \ n\ge 1.
\end{equation}
Then, for $n\ge 1$,
\begin{align}\label{trace}
	\text{tr}(\langle M\rangle_n)
	&=\sum_{k=1}^{n}\mathbb{E}((\Delta M_k)^t \Delta M_k |\mathcal{F}_{k-1})\nonumber\\
	&=\sum_{k=1}^{n}\mathbb{E}(\parallel \Delta M_k\parallel^2 |\mathcal{F}_{k-1}).
\end{align}

Now, for $k\ge 1$, we have
\begin{align}\label{DeltaM_k+1}
	\Delta M_{k+1}
	&= S_{k+1}-\mu W_{k+1} -(S_{k}-\mu W_k)\nonumber\\
	&= T_{k+1}-\mu X_{k+1}\nonumber\\
	&= A_{k+1}X_{\beta(k)}(Y_{k+1}-\mu),
\end{align}
which implies 
\begin{equation}\label{norm_Mk}
	\parallel \Delta M_{k+1}\parallel^2 =(Y_{k+1}-\mu)^2 X_{\beta(k)}^t A_{k+1}^t A_{k+1} X_{\beta(k)}.
\end{equation}
For any $G\in\mathcal{F}_k$, by using \eqref{norm_Mk}, we get
\begin{align}
	\int_G\mathbb{E}(\parallel \Delta M_{k+1}\parallel^2 |\mathcal{F}_k)\,\mathrm{d}\mathbb{P}
	&=\int_G\mathbb{E}((Y_{k+1}-\mu)^2 X_{\beta(k)}^t A_{k+1}^t A_{k+1} X_{\beta(k)} |\mathcal{F}_k)\,\mathrm{d}\mathbb{P}\nonumber\\
	&=\mathbb{E}(((Y_{k+1}-\mu)^2) X_{\beta(k)}^t A_{k+1}^t A_{k+1} X_{\beta(k)}\mathbb{I}_G)\nonumber\\
	&=\mathbb{E}(Y_{k+1}-\mu)^2\mathbb{E}(X_{\beta(k)}^t A_{k+1}^t A_{k+1} X_{\beta(k)}\mathbb{I}_G)\label{indep1}\\
	&=\eta^2 \sum_{l=1}^{k}\mathbb{P}\{\beta(k)=l\}\mathbb{E}(X_{\beta(k)}^t A_{k+1}^t A_{k+1} X_{\beta(k)}\mathbb{I}_G|\beta(k)=l)\nonumber\\
	&= \frac{\eta^2}{k}\sum_{l=1}^{k}\mathbb{E}(X_{l}^t A_{k+1}^t A_{k+1} X_l\mathbb{I}_G)\label{indep2}\\
	&=\frac{\eta^2}{k}\sum_{l=1}^{k}\mathbb{P}\{A_{k+1}=I_d\}\mathbb{E}(X_l^t A_{k+1}^t A_{k+1}X_l \mathbb{I}_G|A_{k+1}=I_d)\nonumber\\
	&\ \ +\frac{\eta^2}{k}\sum_{l=1}^{k}\mathbb{P}\{A_{k+1}=-I_d\}\mathbb{E}(X_l^t A_{k+1}^t A_{k+1}X_l \mathbb{I}_G|A_{k+1}=-I_d)\nonumber\\
	&\ \ +\frac{\eta^2}{k}\sum_{l=1}^{k}\sum_{i=1}^{d-1}\mathbb{P}\{A_{k+1}=J_d^i\}\mathbb{E}(X_l^t A_{k+1}^t A_{k+1}X_l \mathbb{I}_G| A_{k+1}=J_d^i)\nonumber\\
	&\ \ +\frac{\eta^2}{k}\sum_{l=1}^{k}\sum_{i=1}^{d-1}\mathbb{P}\{A_{k+1}=-J_d^i\}\mathbb{E}(X_l^t A_{k+1}^t A_{k+1}X_l \mathbb{I}_G|A_{k+1}= -J_d^i)\nonumber\\
	&=\frac{\eta^2}{k}\sum_{l=1}^{k}\Big(p+(2d-1)\frac{1-p}{2d-1}\Big)\mathbb{E}(\parallel X_l\parallel^2\mathbb{I}_G)\nonumber\\
	&=\frac{\eta^2}{k}\sum_{l=1}^{k}\int_G \parallel X_l\parallel^2\,\mathrm{d}\mathbb{P},\label{last}
\end{align}
where \eqref{indep1} follows from the independence of $Y_{k+1}$ with $\beta(k)$, $A_{k+1}$ and $\{X_1,X_2,\dots,X_k\}$. Similarly, \eqref{indep2} follows on using the independence of $\beta(k)$ with $A_{k+1}$ and $\{X_1,X_2,$ $\dots,X_k\}$. Moreover, the independence of $A_{k+1}$ and $\{X_1,X_2,\dots,X_k\}$ is used to get the penultimate step.

From \eqref{last}, we conclude that
\begin{equation}\label{M2onwards}
	\mathbb{E}(\parallel \Delta M_{k+1}\parallel^2 |\mathcal{F}_k)=\eta^2, \ k\ge 1.
\end{equation}

Let $\mathcal{F}_0=\{\phi, \Omega\}$. Then, from \eqref{first-step} and \eqref{Delta_Mn}, we have
\begin{equation*}
	\Delta M_1=M_1=X_1(Y_1-\mu),
\end{equation*}
which implies
\begin{equation}\label{M1}
	\mathbb{E}(\parallel \Delta M_1\parallel^2|\mathcal{F}_0)=\eta_1^2+(\mu-\mu_1)^2.
\end{equation}
Finally, on substituting \eqref{M2onwards} and \eqref{M1} in \eqref{trace}, we get
\begin{equation}\label{final-trace}
	\text{tr}(\langle M\rangle_n)=\eta_1^2+(\mu-\mu_1)^2+ (n-1)\eta^2, \ n\ge 1.
\end{equation}
\subsection{Moves in MERW with random step sizes} 
Here, we study the number of moves in MERW with random step sizes. In this case, the time instants at which the step size is zero, that is, the walker remains at its current position are referred as delays. By the construction of model, we note that the number of moves till $n$-th step is at least one.
	
Let $\mathcal{Z}_n^*$ and $\mathcal{Z}_n$ denote the number of moves and delays till $n$-th step, respectively. Then, for all $n\ge 1$, we have $\mathcal{Z}_n^* + \mathcal{Z}_n=n.$
Also, 
	\begin{align}\label{N_ones}
		\mathcal{Z}_n^*&=\sum_{k=1}^{n}\mathbb{I}_{\{Y_k\ne  0\}},\ n\ge 1\\
		&=1+\sum_{k=2}^{n}\mathbb{I}_{\{Y_k\ne  0\}}.\nonumber
	\end{align}
For $k\ge 2$, let $Y_k=0$ with probability $b$.	
	\begin{lemma}
		For $n\ge 1$, let $N_n=\mathcal{Z}_n^* -(n-1)(1-b)$ and $\mathcal{F}_n=\sigma(T_1,T_2,\dots,T_n)$. Then, $\{N_n,\mathcal{F}_n\}_{n\ge 1}$ is a martingale.
	\end{lemma}
	
	\begin{proof}
		For $n\ge 1$, we have
		\begin{align*}
			\mathbb{E}(N_{n+1}|\mathcal{F}_n)
			&=\mathbb{E}(\mathcal{Z}_n^*+\mathbb{I}_{\{Y_{n+1}\ne  0\}}-n(1-b)|\mathcal{F}_n)\\
			&=\mathcal{Z}_n^*+\mathbb{E}(\mathbb{I}_{\{Y_{n+1}\ne  0\}})-n(1-b)=N_n,
		\end{align*}
		where the penultimate step follows on using the independence of $Y_{n+1}$ and $\{T_1,T_2,\dots,T_n\}$.\\
		This completes the proof.
	\end{proof}
	
Let us consider the following martingale differences:
	\begin{equation}\label{DeltaN_n}
		\Delta N_n=N_n-N_{n-1}, \ n\ge 1,
	\end{equation}
	with $N_0=0$.
	
	From \eqref{DeltaN_n}, we have
	\begin{equation*}
		N_n=\sum_{k=1}^{n}\Delta N_k, \ n\ge 1.
	\end{equation*}
Next, we obtain the LLN for $\{\mathcal{Z}_n^*\}_{n\ge1}$.
	\begin{theorem}
Let $b=\mathbb{P}\{Y_k=0\}$, $k\ge2$. Then, $\lim_{n\to \infty}\mathcal{Z}_n^*/n=1-b$ a.s.
	\end{theorem}
	
	\begin{proof}
We have 
		\begin{equation}\label{Npow1}
			\mathbb{E}(\Delta N_{n+1}|\mathcal{F}_n)=0, \ n\ge 1.
		\end{equation}
From \eqref{N_ones} and \eqref{DeltaN_n}, we get
	\begin{equation}\label{DeltaN_n+1}
		\Delta N_{n+1}=\mathbb{I}_{\{Y_{n+1}\ne  0\}}-(1-b), \ n\ge 1.
\end{equation}
Thus, 
\begin{align}\label{Npow2}
\mathbb{E}((\Delta N_{n+1})^2|\mathcal{F}_n)
&=\mathbb{E}((\mathbb{I}_{\{Y_{n+1}\ne  0\}}-(1-b))^2 )\nonumber\\
&=(2b-1)\mathbb{E}(\mathbb{I}_{\{Y_{n+1}\ne  0\}})+(1-b)^2\nonumber\\
&=b(1-b).
\end{align}
From \eqref{Npow1}, \eqref{Npow2} and Theorem \ref{thm_LLN2}, we obtain
\begin{equation*}
\lim_{n\to \infty}\frac{1}{n}(\mathcal{Z}_n^*-(n-1)(1-b))=0 \ \text{a.s.},
\end{equation*}
which gives to the required result.	
\end{proof}	
\begin{theorem}
Let $b=\mathbb{P}\{Y_k=0\}$, $k\ge2$. Then, we have the following QSL for $\{\mathcal{Z}_n^*\}_{n\ge1}$:
\begin{equation*}
\lim_{n\to\infty}\frac{1}{\log (n+1)}\sum_{k=1}^{n}\frac{(\mathcal{Z}_k^*-(k-1)(1-b))^2}{k(k+1)}=b(1-b) \ \text{a.s.}
\end{equation*}
\end{theorem}
\begin{proof}
From \eqref{DeltaN_n+1}, we have $(\Delta N_{n+1})^2\le 1$ for all $n\ge 1$, and $\Delta N_1=1$. Thus, 
\begin{equation*}
\sup_{n\ge 0}\mathbb{E}(\Delta N_{n+1})^4|\mathcal{F}_n)< \infty.
\end{equation*}
Let $\Phi_n=1$ for all $n\ge 0$ in Theorem \ref{quadLaw}. So, $v_n=n+1$, and hence $v_\infty=\infty$. Also, the explosion coefficient associated with $\{\Phi_n\}_{n\ge 0}$ is $f_n=1/(n+1)$ which tends to 0 as $n\to\infty$. By using Theorem \ref{quadLaw} and \eqref{Npow2}, we obtain
\begin{equation*}
\lim_{n\to\infty}\frac{1}{\log (n+1)}\sum_{k=1}^{n}\frac{N_k^2}{k(k+1)}=b(1-b) \ \text{a.s.},
\end{equation*}
which yields the required result.
\end{proof}	
\begin{theorem}
Let $b=\mathbb{P}\{Y_k=0\}$, $k\ge2$. Then, the following LIL for $\{\mathcal{Z}_n^*\}_{n\ge1}$ holds true:
\begin{equation*}
\limsup_{n\to\infty}\frac{|\mathcal{Z}_n^*-(n-1)(1-b)|}{\sqrt{2n\log\log n}}\le \sqrt{b(1-b)} \ \text{a.s.}
\end{equation*}
\end{theorem}
\begin{proof}
For $n\ge 1$, from \eqref{DeltaN_n+1}, we have
\begin{equation*}
\mathbb{E}(|\Delta N_{n+1}|^3|\mathcal{F}_n)	=\mathbb{E}\big(|\mathbb{I}_{\{Y_{n+1}\ne  0\}}-(1-b)|^3\big)
			\le 1.
\end{equation*}
Consequently,
		\begin{equation}\label{supNn}
			\sup_{n\ge1}\mathbb{E}(|\Delta N_{n+1}|^3|\mathcal{F}_n)< \infty.
		\end{equation}
On taking $U_n=1$, $n\ge0$ in Theorem \ref{LIL}, we have $\tau_n=n+1$. So,
		\begin{equation}\label{vn_conv}
			\sum_{n=0}^{\infty}\tau_n^{-3/2}<\infty.
		\end{equation}		
		By using \eqref{Npow1}, \eqref{Npow2}, \eqref{supNn}, \eqref{vn_conv} and Theorem \ref{LIL}, we obtain the required result.
	\end{proof}
Next, we provide a CLT related to the number of moves in MERW with random step sizes.
\begin{theorem}
Let $b=\mathbb{P}\{Y_k=0\}$, $k\ge2$. Then, $(\mathcal{Z}_n^*-(n-1)(1-b))/\sqrt{n}\xrightarrow{d}\mathcal{N}(0,b(1-b))$.
\end{theorem}
\begin{proof}
	The predictable square variation of $\{N_n,\mathcal{F}_n\}_{n\ge 1}$ is given by
	\begin{equation}\label{quadN}
		\langle N\rangle_n=\sum_{k=1}^{n}\mathbb{E}((\Delta N_k)^2|\mathcal{F}_{k-1}).
	\end{equation}
As $\mathbb{E}((\Delta N_1)^2|\mathcal{F}_{0})=1$, by using \eqref{Npow2} and \eqref{quadN}, we obtain $\langle N\rangle_n=1+(n-1)b(1-b).$ Thus,
	\begin{equation}\label{c1}
		\lim_{n\to\infty}\frac{\langle N\rangle_n}{n}=b(1-b).
	\end{equation}
	From \eqref{DeltaN_n+1}, we get $|\Delta N_k|\le 1$ for all $k\ge 1$. So, for all $\epsilon>0$, we have
	\begin{equation}\label{c2}
		\lim_{n\to\infty}\frac{1}{n}\sum_{k=1}^{n}\mathbb{E}((\Delta N_k)^2\mathbb{I}_{\{|\Delta N_k|\ge \epsilon\sqrt{n}\}}|\mathcal{F}_{k-1})=0 \ \text{a.s.}
	\end{equation}
Now, by using \eqref{c1} and \eqref{c2}, the required result follows as a consequence of Theorem \ref{CLT}.
\end{proof}

\subsection{Law of large numbers}
Here, we state and prove some LLN-type results for MERW with random step sizes $\{S_n\}_{n\ge1}$ using a martingale approach. 
\begin{theorem}
Let $\alpha>1/2$ and $0\le p< (2d+1)/4d$. Then, 
$\lim_{n\to \infty} S_n/n^{\alpha}=0$ a.s.
\end{theorem}
\begin{proof}
Let us denote $s_n=\text{tr}(\langle M\rangle_n)$. From \eqref{final-trace}, we have $s_\infty=\infty$. Also, 
	\begin{equation}\label{s_n_equiv}
		s_n \sim n\eta^2.
	\end{equation}
For $\epsilon>0$, by using Theorem \ref{LLN}, we have
	\begin{equation}\label{llnForm}
		\frac{\parallel M_n\parallel^2}{\lambda_{\max}(\langle M\rangle_n)}=o\big((\log s_n)^{1+\epsilon}\big) \ \text{a.s.}
	\end{equation}
As $\lambda_{\max}(\langle M\rangle_n)\le s_n$, from \eqref{s_n_equiv} and \eqref{llnForm}, we obtain
	\begin{equation}\label{eta2}
		\parallel M_n\parallel^2 =o\big(n\eta^2 (\log n\eta^2)^{1+\epsilon}\big) \ \text{a.s.}
	\end{equation}
	Further,
	\begin{equation}\label{ord-2alpha}
		\log n\eta^2 \sim \log n \ \text{and} \ n(\log n)^{1+\epsilon}=o(n^{2\alpha}). 
	\end{equation}
	Thus, by using \eqref{ord-2alpha} in \eqref{eta2}, we get $\parallel M_n\parallel^2=o(n^{2\alpha})$ a.s., 
	which implies that
	\begin{equation}\label{penul}
		\lim_{n\to \infty} \frac{M_n}{n^\alpha}=0 \ \text{a.s.}
	\end{equation}
	Recall that $M_n=S_n-\mu W_n$. By using Remark 3.1 of Bercu and Laulin (2019) in \eqref{penul}, we get the required result.
\end{proof}

\begin{theorem}
Let $\alpha>1/2$ and $p=(2d+1)/4d$. Then,
	\begin{equation}\label{lgalpsn}
		\lim_{n\to \infty}\frac{S_n}{\sqrt{n}(\log n)^\alpha}= 0 \ \text{a.s.}
	\end{equation}
\end{theorem}
\begin{proof}
In \eqref{eta2}, we use $\log n\eta^2 \sim \log n$ as $n \to \infty$ to obtain $\parallel M_n\parallel^2=o\big(n(\log n)^{2\alpha}\big)$ a.s. That is,
	\begin{equation}\label{penul1}
		\lim_{n\to \infty}\frac{M_n}{\sqrt{n}(\log n)^\alpha}=0 \ \text{a.s.}
	\end{equation}
By using Remark 3.3 of Bercu and Laulin (2019) in \eqref{penul1}, we get \eqref{lgalpsn}.	
\end{proof}
The next result is discussed in Theorem 3.1 of Zhang (2024) using a recursive stochastic algorithm. Here, we provide an alternate proof.
\begin{theorem}
Let $(2d+1)/4d<p\le 1$ and $\gamma=(2dp-1)/(2d-1)$. Then, 
	\begin{equation*}
		\lim_{n\to \infty}\frac{S_n}{n^\gamma}=\mu S \ \text{a.s.},
	\end{equation*}
	 where $S$ is a non-degenerate random vector.
\end{theorem}
\begin{proof}
Note that \eqref{penul} holds for all $\alpha>1/2$ and $0\le p\le 1$.
Also, $p>(2d+1)/4d$ iff $\gamma>1/2$. Thus, from \eqref{penul}, we have
	\begin{equation}\label{penul2}
		\lim_{n\to \infty}\frac{1}{n^\gamma}(S_n-\mu W_n)=0 \ \text{a.s.}
	\end{equation}
By using Theorem 3.7 of Bercu and Laulin (2019), we get
	\begin{equation}\label{thm3.7B}
		\lim_{n\to \infty}\frac{W_n}{n^\gamma}=S \ \text{a.s.},
	\end{equation}
	where $S$ is a non-degenerate random vector.
	Finally, from \eqref{penul2} and \eqref{thm3.7B}, the required result follows.
\end{proof}	
\subsection{Quadratic strong law}\label{qslstep}
In order to derive the result related to QSL, we consider the inner product  $\langle u,v\rangle=u^t v, \ u\in\mathbb{R}^d, \, v\in\mathbb{R}^d.$
So, $|\langle u,v\rangle|^2=\langle u,v\rangle^2=u^tvv^tu.$

For $n\ge 1$, let $\epsilon_n=\Delta M_n$ be the martingale difference given in \eqref{Delta_Mn}. For any vector $u\in\mathbb{R}^d$, let $M_n(u)=\langle u,M_n\rangle=u^t M_n$ and $\epsilon_n(u)=\langle u,\epsilon_n\rangle=u^t\epsilon_n$.
Then, $\{M_n(u),\mathcal{F}_n\}_{n\ge 1}$ is a real martingale. 
	
For $u\in\mathbb{R}^d$, we have
	\begin{equation}\label{eps_uPow2}
	\mathbb{E}(|\epsilon_{n+1}(u)|^2|\mathcal{F}_n)
	=\mathbb{E}(u^t\epsilon_{n+1}\epsilon_{n+1}^tu|\mathcal{F}_n)=u^t \mathbb{E}(\epsilon_{n+1}\epsilon_{n+1}^t|\mathcal{F}_n)u.
\end{equation}
From \eqref{DeltaM_k+1}, we get
\begin{align}\label{eps_eps^t}
	\mathbb{E}(\epsilon_{n+1}\epsilon_{n+1}^t|\mathcal{F}_n)
	&=\mathbb{E}((Y_{n+1}-\mu)^2)\,\mathbb{E}(A_{n+1}X_{\beta(n)}X_{\beta(n)}^t A_{n+1}^t|\mathcal{F}_n)\nonumber\\
	&=\eta^2 \mathbb{E}(X_{n+1}X_{n+1}^t|\mathcal{F}_n),
\end{align}
where the last step follows from \eqref{multstep}.
	
	From Eq. (4.5) of Bercu and Laulin (2019), we have 
	\begin{equation}\label{XnXn^t}
		\mathbb{E}(X_{n+1}X_{n+1}^t|\mathcal{F}_n)=\frac{\gamma}{n}\Sigma_n + \frac{1-\gamma}{d}I_d,
	\end{equation}
	where $\gamma=(2dp-1)/(2d-1)$, $\Sigma_n= \sum_{i=1}^{d}\sum_{k=1}^{n}\mathbb{I}_{\{X_k^{(i)}\ne 0\}} e_i e_i^t$ and $X_k^{(i)}$ is the $i$-th coordinate of $X_k$.
	
	Also, from  Eq. (5.3) of Bercu and Laulin (2019), we have
	\begin{equation}\label{Sigma_nConv}
		\lim_{n\to \infty}\frac{\Sigma_n}{n}=\frac{I_d}{d} \ \text{a.s.}
	\end{equation}
	By using \eqref{XnXn^t} and \eqref{Sigma_nConv} in \eqref{eps_eps^t}, we get
	\begin{equation}\label{lim_u}
		\lim_{n\to \infty}\mathbb{E}(\epsilon_{n+1}\epsilon_{n+1}^t|\mathcal{F}_n)=\frac{\eta^2}{d}I_d \ \text{a.s.}
	\end{equation}
	Thus, from \eqref{eps_uPow2} and \eqref{lim_u}, we obtain
	\begin{equation}\label{limPow2}
		\lim_{n\to \infty}\mathbb{E}(|\epsilon_{n+1}(u)|^2|\mathcal{F}_n)= \frac{\eta^2}{d}\parallel u\parallel^2 \ \text{a.s.}
	\end{equation}
	By using the Cauchy-Schwarz inequality and \eqref{DeltaM_k+1}, we have
	\begin{align}\label{pow_3_be}
		\mathbb{E}(|\epsilon_{n+1}(u)|^3|\mathcal{F}_n)
		&\le\parallel u\parallel^3 \mathbb{E}(|Y_{n+1}-\mu|^3 (X_{\beta(n)}^t A_{n+1}^t A_{n+1}X_{\beta(n)})^{3/2}|\mathcal{F}_n)\nonumber\\
		&=\parallel u\parallel^3 \mathbb{E}(|Y_{n+1}-\mu|^3) \mathbb{E}((X_{\beta(n)}^t A_{n+1}^t A_{n+1}X_{\beta(n)})^{3/2}|\mathcal{F}_n),
	\end{align}
	where the last step follows on using the independence of $Y_{n+1}$ with $\beta(n)$, $A_{n+1}$ and $\{X_1,X_2$, $\dots,X_n\}$.
	
	For any $G\in\mathcal{F}_n$, we have
	\begin{align}
	\int_G\mathbb{E}((X_{\beta(n)}^t A_{n+1}^t &A_{n+1}X_{\beta(n)})^{3/2}|\mathcal{F}_n)\,\mathrm{d}\mathbb{P}\nonumber\\
	&=\mathbb{E}((X_{\beta(n)}^t A_{n+1}^t A_{n+1} X_{\beta(n)})^{3/2}\mathbb{I}_G)\nonumber\\
	&=\sum_{k=1}^{n}\mathbb{P}\{\beta(n)=k\}\mathbb{E}((X_{\beta(n)}^t A_{n+1}^t A_{n+1} X_{\beta(n)})^{3/2}\mathbb{I}_G|\beta(n)=k)\nonumber\\
	&= \frac{1}{n}\sum_{k=1}^{n}\mathbb{E}((X_{k}^t A_{n+1}^t A_{n+1} X_k)^{3/2}\mathbb{I}_G)\label{indep3}\\
	&=\frac{1}{n}\sum_{k=1}^{n}\mathbb{P}\{A_{n+1}=I_d\}\mathbb{E}((X_{k}^t A_{n+1}^t A_{n+1} X_k)^{3/2} \mathbb{I}_G|A_{n+1}=I_d)\nonumber\\
	&\ \ +\frac{1}{n}\sum_{k=1}^{n}\mathbb{P}\{A_{n+1}=-I_d\}\mathbb{E}((X_{k}^t A_{n+1}^t A_{n+1} X_k)^{3/2} \mathbb{I}_G|A_{n+1}=-I_d)\nonumber\\
		&\ \ +\frac{1}{n}\sum_{k=1}^{n}\sum_{i=1}^{d-1}\mathbb{P}\{A_{n+1}=J_d^i\}\mathbb{E}((X_{k}^t A_{n+1}^t A_{n+1} X_k)^{3/2} \mathbb{I}_G| A_{n+1}=J_d^i)\nonumber\\
		&\ \ +\frac{1}{n}\sum_{k=1}^{n}\sum_{i=1}^{d-1}\mathbb{P}\{A_{n+1}=-J_d^i\}\mathbb{E}((X_{k}^t A_{n+1}^t A_{n+1} X_k)^{3/2} \mathbb{I}_G|A_{n+1}= -J_d^i)\nonumber\\
		&=\frac{1}{n}\sum_{k=1}^{n}\Big(p+(2d-1)\frac{1-p}{2d-1}\Big)\mathbb{E}(\parallel X_k\parallel^3\mathbb{I}_G)\nonumber\\
		&=\frac{1}{n}\sum_{k=1}^{n}\int_G \mathrm{d}\mathbb{P},\label{last1}
	\end{align}
	where \eqref{indep3} follows on using the independence of $\beta(n)$ with $A_{n+1}$ and $\{X_1,X_2,\dots,X_n\}$. Moreover, the independence of $A_{n+1}$ and $\{X_1,X_2,\dots,X_n\}$ is used to get the penultimate step. Thus, from \eqref{last1}, we get
	\begin{equation}\label{pow_3/2}
		\mathbb{E}((X_{\beta(n)}^t A_{n+1}^t A_{n+1}X_{\beta(n)})^{3/2}|\mathcal{F}_n)=1, \ n\ge 1.
	\end{equation}
	By substituting \eqref{pow_3/2} in \eqref{pow_3_be}, we obtain
	\begin{equation}\label{pow_3}
		\mathbb{E}(|\epsilon_{n+1}(u)|^3|\mathcal{F}_n)\le \parallel u\parallel^3 \mathbb{E}(|Y_{n+1}-\mu|^3), \ n\ge 1.
	\end{equation}
	Also, we have
	\begin{equation}\label{eps1Pow3}
		\mathbb{E}(|\epsilon_{1}(u)|^3|\mathcal{F}_0)
		\le \parallel u\parallel^3 \mathbb{E}(|Y_1-\mu|^3).
	\end{equation}
	The proof of the following lemma follows from \eqref{pow_3} and \eqref{eps1Pow3}.
	\begin{lemma}\label{sup_epsPow3}
		Let the random step sizes $\{Y_n\}_{n\ge 1}$ in MERW $\{S_n\}_{n\ge1}$ be such that their third absolute moments about $\mu$ are finite, that is, $\mathbb{E}(|Y_n-\mu|^3)<\infty$, for all $n\ge1$. Then, $\sup_{n\ge 0}\mathbb{E}(|\epsilon_{n+1}(u)|^3|\mathcal{F}_n)<\infty$.
	\end{lemma}
	\begin{theorem}
		Under the assumptions of Lemma
		\ref{sup_epsPow3}, the following result holds:
		\begin{equation*}
			\lim_{n\to\infty}\frac{1}{\log (n+1)}\sum_{k=1}^{n} \frac{(S_k-\mu W_k) (S_k^t-\mu W_k^t)}{k(k+1)}=\frac{\eta^2}{d} \ \text{a.s.}
		\end{equation*}
	\end{theorem}
	
	\begin{proof}
		For $u\in\mathbb{R}^d$,	by using \eqref{limPow2}, Lemma \ref{sup_epsPow3} and Theorem \ref{quadLaw}, we obtain
		\begin{equation*}
			\lim_{n\to\infty}\frac{1}{\log (n+1)}\sum_{k=1}^{n} \frac{1}{k(k+1)}(M_k(u))^2=\frac{\eta^2}{d}\parallel u\parallel^2 \ \text{a.s.}
		\end{equation*}
		Equivalently,
		\begin{equation}\label{penul6}
			\lim_{n\to\infty}\frac{1}{\log (n+1)}\sum_{k=1}^{n} \frac{1}{k(k+1)}u^t M_k M_k^t u=\frac{\eta^2}{d}\parallel u\parallel^2 \ \text{a.s.}
		\end{equation}
		Finally, by using \eqref{penul6} and Proposition \ref{quadform}, we get the required result.	
	\end{proof}
	
	\subsection{Law of iterated logarithm}
	In order to derive the LIL-type results, we continue with the same inner product as considered in Section \ref{qslstep}. Recall that $\epsilon_n=\Delta M_n$, $M_n(u)=\langle u,M_n\rangle=u^t M_n$ and $\epsilon_n(u)=\langle u,\epsilon_n\rangle=u^t\epsilon_n$ for all $n\ge 1$, $u\in\mathbb{R}^d$.
	Also, $\{M_n(u),\mathcal{F}_n\}_{n\ge 1}$ is a real martingale.
	
	By using Cauchy-Schwarz inequality, we get
	\begin{equation}\label{Cau_Sch}
		|\epsilon_n(u)|=|\langle u,\epsilon_n\rangle|\le \parallel u\parallel \parallel \epsilon_n\parallel, \ n\ge 1.
	\end{equation}
	From \eqref{Cau_Sch}, we have 
	\begin{align}\label{pow_2}
		\mathbb{E}(|\epsilon_{n+1}(u)|^2|\mathcal{F}_n)&\le \parallel u\parallel^2 \mathbb{E}(\parallel \epsilon_{n+1}\parallel^2|\mathcal{F}_n)\nonumber \\
		&= \eta^2 \parallel u\parallel^2, \ n\ge 1,
	\end{align}
	where the last step follows on using \eqref{M2onwards}.

	\begin{lemma}\label{lil_Mn(u)}
		Let the random step sizes $\{Y_n\}_{n\ge 1}$ in MERW $\{S_n\}_{n\ge1}$ be such that their third absolute moments about $\mu$ are finite, that is, $\mathbb{E}(|Y_n-\mu|^3)<\infty$, for all $n\ge1$. Then, we have \vspace{.1cm}\\
		\noindent{(i)} $\limsup_{n\to\infty}\frac{|M_n(u)|}{\sqrt{2n\log\log n}}\le \eta \parallel u\parallel \ \text{a.s.}$\vspace{.1cm}\\
		\noindent{(ii)} $\limsup_{n\to\infty}\frac{\parallel M_n\parallel}{\sqrt{2n\log\log n}}\le \eta \sqrt{d} \ \text{a.s.}$
\end{lemma}
\begin{proof}
		Let $\Delta M_n(u)=M_n(u)-M_{n-1}(u)$, $n\ge1$ with $M_0(u)=0$. Then, $\epsilon_n(u)=\Delta M_n(u)$ and
		\begin{equation}\label{pow_1}
			\mathbb{E}(\epsilon_{n+1}(u)|\mathcal{F}_n)=0.
		\end{equation}
		Also, we have
		\begin{equation}\label{pow_01}
			\mathbb{E}(\epsilon_{1}(u)|\mathcal{F}_0)=0.
		\end{equation}		
		Note that the first two conditions of Theorem \ref{LIL} follow from \eqref{pow_3}, \eqref{eps1Pow3}, \eqref{pow_2}, \eqref{pow_1} and \eqref{pow_01}.
		
		Also, we have $M_n(u)=\sum_{k=1}^{n}\epsilon_k(u)$. To apply Theorem \ref{LIL}, we need to verify its remaining hypothesis. We take $\Phi_n=U_n=1$ for all $n\ge 0$. So, $\tau_n=n+1$, and therefore $\tau_\infty=\infty$.
		Moreover,
		\begin{equation}
			\sum_{n=0}^{\infty}U_n^3 \tau_n^{-3/2}=\sum_{n=0}^{\infty}\frac{1}{(n+1)^{3/2}}<\infty.
		\end{equation}
		Thus, the first part of the result is an immediate consequence of Theorem \ref{LIL}.
	
For the second part, we proceed as follows: From Lemma \ref{lil_Mn(u)}(i), we have
\begin{equation}\label{foru}
	\limsup_{n\to\infty}\frac{(M_n(u))^2}{2n\log\log n}\le \eta^2 \parallel u\parallel^2 \ \text{a.s.}
\end{equation}
On taking $u=e_i$, $1\le i\le d$ in \eqref{foru}, we get
\begin{equation}\label{fore_i}
	\limsup_{n\to\infty}\frac{\langle e_i,M_n\rangle^2}{2n\log\log n}\le \eta^2 \ \text{a.s.}
\end{equation}
As $\{e_1,e_2,\dots,e_d\}$ is an orthonormal basis of $\mathbb{R}^d$, it follows from Parseval's identity that
\begin{equation}\label{Parseval}
	\parallel M_n\parallel^2=\sum_{i=1}^{d}\langle e_i,M_n\rangle^2.
\end{equation}
From \eqref{Parseval}, we have
\begin{equation}\label{penul3}
	\limsup_{n\to\infty}\frac{\parallel M_n\parallel^2}{2n\log \log n}\le \sum_{i=1}^{d}\limsup_{n\to\infty}\frac{\langle e_i,M_n\rangle^2}{2n\log\log n} \ \text{a.s.}
\end{equation}
Finally, on using \eqref{fore_i} in \eqref{penul3}, we obtain
\begin{equation*}
	\limsup_{n\to\infty}\frac{\parallel M_n\parallel^2}{2n\log \log n}\le d\eta^2 \ \text{a.s.}
\end{equation*}
This completes the proof.	
\end{proof}
\begin{theorem}\label{zh22}
Under the assumptions of Lemma \ref{lil_Mn(u)}, the following results hold true for $\{S_n\}_{n\ge1}$:\vspace{0.1cm}\\
	\noindent {(i)} For $0\le p< (2d+1)/4d$, we have
	\begin{equation*}
		\limsup_{n\to\infty}\frac{\parallel S_n\parallel}{\sqrt{2n\log\log n}}\le \eta \sqrt{d}+\sqrt{\frac{\mu^2(2d-1)}{1+2d(1-2p)}} \ \text{a.s.}
	\end{equation*}
	\noindent {(ii)} For $p= (2d+1)/4d$, we have
	\begin{equation*}
		\limsup_{n\to\infty}\frac{\parallel S_n\parallel}{\sqrt{2n\log n \log\log\log n}}\le |\mu| \ \text{a.s.}
	\end{equation*}
\end{theorem}

\begin{proof}
To prove the first part, we recall that $M_n=S_n-\mu W_n$. So,
	\begin{equation}\label{penul4}
		\limsup_{n\to\infty}\frac{\parallel S_n\parallel}{\sqrt{2n\log\log n}}\le \limsup_{n\to\infty}\frac{\parallel M_n\parallel}{\sqrt{2n\log\log n}} + |\mu| \limsup_{n\to\infty}\frac{\parallel W_n\parallel}{\sqrt{2n\log\log n}} \ \text{a.s.}
	\end{equation}
By using Eq. (3.4) of Bercu and Laulin (2019) together with Lemma \ref{lil_Mn(u)}(ii) in \eqref{penul4}, we get the required result.

To establish the second part, we proceed similar to the proof of first part and obtain the following inequality:
	{\footnotesize	\begin{equation}\label{penul5}
			\limsup_{n\to\infty}\frac{\parallel S_n\parallel}{\sqrt{2n\log n \log\log\log n}}\le \limsup_{n\to\infty}\frac{\parallel M_n\parallel}{\sqrt{2n\log n \log\log\log n}} + |\mu| \limsup_{n\to\infty}\frac{\parallel W_n\parallel}{\sqrt{2n\log n \log\log\log n}} \ \text{a.s.}
	\end{equation}}
From Lemma \ref{lil_Mn(u)}(ii), it follows that
	\begin{equation}\label{zero}
		\limsup_{n\to\infty}\frac{\parallel M_n\parallel}{\sqrt{2n\log n \log\log\log n}}=0 \ \text{a.s.}
	\end{equation}
On using Eq. (3.9) of Bercu and Laulin (2019) and \eqref{zero} in \eqref{penul5}, the required result follows. 
\end{proof}
\begin{remark}
For the exact limits of the quantities involved in Theorem \ref{zh22}, we refer the reader to Corollary 3.1 of Zhang (2024).
\end{remark}

\section{Concluding remarks}
First, we studied the number of moves in the MERW with stops. We derived the conditional mean increment of the number of moves and used this result to establish a recursive relation. Solving this recursion yields an explicit expression for the expected number of moves in the MERW with stops. Furthermore, we proved several almost sure convergence results, including the law of large numbers and the law of the iterated logarithm, under different parameter regimes. Later, we discussed the MERW with random step sizes using a martingale approach. In this case, first, we discussed about the number of moves of the walk. A suitable martingale is constructed for the number of moves which is then used to establish the law of large numbers, the quadratic strong law, the law of the iterated logarithm, and the central limit theorem for the number of moves in the MERW with random step sizes. Additionally, the law of large numbers is established in three distinct regimes that are determined by the value of memory parameter, and the quadratic strong law type result is derived for the walk itself. Moreover, the law of iterated logarithm is obtained for two different regimes of the MERW with random step sizes. The results obtained in this paper complement those of Bercu and Laulin (2019), Zhang (2024) and Bercu (2025).

\section*{Acknowledgement}
The second author thanks Government of India for the grant of Prime Minister's Research Fellowship, ID 1003066. 	
		
\end{document}